\documentclass[12pt]{amsart}
\usepackage[dvips]{graphicx}
\usepackage{amsmath,graphics}
\usepackage{amsfonts,amssymb}
\usepackage{xypic}
\usepackage{comment}
\usepackage{color}
\specialcomment{proofc}{}{}
\includecomment{proofc}
\theoremstyle{plain}
\newtheorem*{theorem*}{Theorem}
\newtheorem*{lemma*} {Lemma}
\newtheorem*{corollary*} {Corollary}
\newtheorem*{proposition*}{Proposition}
\newtheorem*{conjecture*}{Conjecture}
\newtheorem{theorem}{Theorem}[section]
\newtheorem{lemma}[theorem]{Lemma}
\newtheorem*{theorem1*}{Theorem 1}
\newtheorem*{theorem2*}{Theorem 2}
\newtheorem*{theorem3*}{Theorem 3}
\newtheorem{corollary}[theorem]{Corollary}
\newtheorem{proposition}[theorem]{Proposition}

\theoremstyle{remark}
\newtheorem*{remark}{Remark}
\newtheorem*{definition}{Definition}

\newtheorem{example*}{Example}
\newtheorem*{claim}{Claim}
\newtheorem*{ack}{Acknowledgement}
\theoremstyle{definition}

\def\G{\Gamma}

  \def\F{\Bbb{F}} \def\Z{\Bbb{Z}} \def\R{\Bbb{R}} 
\def\N{\Bbb{N}}   \def\ll{\langle} \def\rr{\rangle}
 \def\a{\alpha}   \def\bp{\begin{pmatrix}}
\def\sm{\setminus} \def\ep{\end{pmatrix}} \def\bn{\begin{enumerate}} 
   \def\en{\end{enumerate}}
\def\ba{\begin{array}} \def\ea{\end{array}}  
   \def\a{\alpha}

\def\ker{\mbox{Ker}}\def\be{\begin{equation}} \def\ee{\end{equation}} 
   
 \def\hom{\mbox{Hom}}

\begin{document}
\title[On the topology of Symplectic Calabi--Yau $4$--manifolds]{On the topology of Symplectic Calabi--Yau $4$--manifolds}
\author{Stefan Friedl}
\address{Mathematisches Institut\\ Universit\"at zu K\"oln\\   Germany}
\email{sfriedl@gmail.com}

\author{Stefano Vidussi}
\address{Department of Mathematics, University of California,
Riverside, CA 92521, USA} \email{svidussi@math.ucr.edu} \thanks{S. Vidussi was partially supported by NSF grant
\#0906281.}

\begin{abstract}
Let $M$ be a $4$-manifold with residually finite fundamental group $G$ having $b_1(G) > 0$. Assume that $M$ carries a symplectic structure with trivial canonical class $K = 0 \in H^2(M)$. Using a theorem of Bauer and Li, together with some classical results in $4$--manifold topology, we show that for a large class of groups $M$ is determined up to homotopy  and, in favorable circumstances,  up to homeomorphism by its fundamental group. This is analogous to what was proven by Morgan--Szab\'o in the case of $b_1 = 0$ and provides further evidence to the conjectural classification of symplectic $4$--manifolds with $K = 0$. As a side, we obtain a result that has some independent interest, namely the fact that the fundamental group of a surface bundle over a surface is large, except for the obvious cases.
\end{abstract}
\maketitle


\vspace{-0.55cm}
\subsection{Introduction}
A classification, even conjectural, of symplectic manifolds of (real) dimension $4$ can be considered at the moment out of reach. In this realm we must currently content ourselves with tackling this problem for some limited class of manifolds where the problem becomes more tractable. An example of that, that mimics the approach common in complex geometry, is to study manifolds of a given Kodaira dimension (see \cite{Li06a}). This approach has allowed a complete classification, up to diffeomorphism, in the case of Kodaira dimension $\kappa = - \infty$.
The next step is to tackle the case of $\kappa =0$ where some encouraging results, described in part below, are already available. For sake of presentation, here we will be concerned with the case of trivial canonical class $K = 0$, the difference being almost immaterial, see \cite[Theorem 2.4 and Proposition 6.3]{Li06a}.

Examples of symplectic manifolds with $K = 0$ are familiar, but quite exceptional: in particular, the only known example with trivial fundamental group is the $K3$ surface. Motivated by this fact, Morgan and Szab\'o proved in \cite{MS97} that a simply connected symplectic $4$--manifold with $K = 0$ is  homotopy equivalent (hence homeomorphic, by Freedman's work) to the $K3$ surface. In fact, as remarked in \cite[Corollary 1.4]{Bau08},
their result extends to all manifolds with $b_1 = 0$ as long as the fundamental group has nontrivial finite quotients. From the vantage point of this note, it is convenient to rephrase the main result of \cite{MS97} in terms of fundamental groups.
We begin by introducing the following notation.
\begin{definition} A closed $4$--manifold $M$ is a \textit{Symplectic Calabi--Yau} (SCY for short) manifold if it admits a symplectic structure with $K = 0 \in H^2(M;\Z)$.
A finitely presented group is an \emph{SCY group} if it is the fundamental group of an SCY $4$--manifold. \end{definition}
(Henceforth all groups will be implicitly assumed to be finitely presented.)  We will state a weaker version of the main result of \cite{MS97}, restricting ourselves to residually finite groups, more manageable for our purposes.

\begin{theorem}  \textbf{\emph{(Morgan--Szab\'o)}} \label{thm:morsza} Let $G$ be a residually finite SCY group with $b_1(G) = 0$. Then the following holds true:
\begin{enumerate}
\item $G$ is the trivial group;
\item the corresponding SCY manifolds are homotopy equivalent, hence unique up to homeomorphism.
\end{enumerate}  \end{theorem}

At this point we can ask if results similar to those of Theorem \ref{thm:morsza} hold  for groups  with $b_1(G) > 0$.  The geography of known SCY manifolds with positive first Betti number, as we discuss in Section \ref{sec:examples},  is fairly limited and is composed exclusively by infrasolvmanifolds, all (finitely covered by) torus bundles over a torus. Donaldson \cite{Do08} and Li \cite{Li06a} have suggested the possibility that this list is complete.
This suggestion is supported by work of  Bauer and Li  in \cite{Bau08,Li06b}:

\begin{theorem} \textbf{\emph{(Bauer, Li)}} \label{thm:bauerli}  Let $G$ be a SCY group with $b_1(G) > 0$. Then  the following holds true:
\begin{enumerate}
\item $2 \leq b_{1}(G) \leq 4$;
\item the corresponding SCY manifolds satisfy $\chi(M) = \sigma(M) = 0$.
\end{enumerate} \end{theorem}

If $M$ is a SCY manifold all its finite covers are. The theorem above  entails therefore that $2 \leq b_1(G) \leq vb_1(G) \leq 4$, where $vb_{1}(G) = \mbox{sup}\{ b_{1}(G_{i}) | G_{i} \leq_{f.i.} G\}$.

The purpose of this note is to discuss how to bridge the gap between the statement of Theorem \ref{thm:bauerli} and that of  Theorem \ref{thm:morsza}, i.e. to attempt to give a topological classification for $b_1 > 0$ that mimics parts (1) and (2) of  Theorem \ref{thm:morsza}.

Regarding part (1), our work will be mostly devoted to apply Theorem \ref{thm:bauerli} to  compute certain group--theoretic invariants for SCY groups, to show that they are consistent with those of infrasolvmanifolds groups. The one merit
 is to show that  Theorem \ref{thm:bauerli}  constrains quite effectively SCY groups at least within some interesting classes of groups. In particular we will show, as consequence of  the recent work of Agol and Wise \cite{Ag12,Wi12}, the following alternative, that is possibly of independent interest:

\begin{theorem} Let $G$ be the fundamental group of a surface bundle over a surface $\Sigma_{h} \hookrightarrow M \to \Sigma_{g}$; then $G$ is either large  or $\mbox{max}\,(g,h) \leq 1$. \end{theorem}
(Recall that a group $G$ is large if a finite index subgroup surjects onto a nonabelian free group. In such case, $ vb_1(G) = \infty$.\footnote{
While preparing the final version of this paper we learned that, independently, R. \.{I}. Baykur  (see \cite{Ba12}) and T. J. Li--Y. Ni (see \cite{LNi12}) obtained, under the same assumptions, similar conclusions on the virtual Betti number.})

While it appears very unlikely that the constraints of Theorem \ref{thm:bauerli} characterize SCY groups,  we will show that for the $3$--dimensional analog (namely determining, among all $3$--manifolds, the class of those manifolds $N$ that fiber over $S^1$ with a fibration with Euler class $e = 0  \in H^2(N;\Z)$) the condition $1 \leq b_1(G) \leq vb_1(G) \leq 3$ is ``almost" sufficient to characterize the class.
  (The ``almost" is due to the presence of  $S^1 \times S^2$.)  This is a rather straightforward consequence of highly nontrivial facts on $3$--manifold groups (culminating in \cite{Ag12,Wi12}).

Our strongest
 results are about  part (2) and show, for a large class of SCY groups, the uniqueness of the homotopy type for the corresponding manifold:

\begin{theorem} \label{thm:main} Let $G$ be a residually finite SCY group with $b_1(G) > 0$. Assume $H^2(G;\Z[G]) = 0$. Then the corresponding SCY manifolds are homotopy equivalent  Eilenberg--Maclane spaces $K(G,1)$. \end{theorem}

In the framework of residually finite fundamental groups, Theorem \ref{thm:main} reduces the study up to homotopy of SCY $4$--manifolds with $b_1 > 0$ to determining SCY groups, under the assumption that $H^2(G;\Z[G]) = 0$.
 This assumption holds for large classes of groups, in particular all virtual duality groups of virtual cohomological dimension at least $3$  (see \cite[Proposition~VIII.11.3]{Br94}). Remarkably for us, this is the case for the fundamental groups of all known examples of SCY manifolds, as they are Poincar\'e duality groups of dimension $4$. These groups are  virtually poly--$\Z$ hence residually finite. Also, as for virtually poly--$\Z$ groups the Borel conjecture
  holds true (see \cite{FJ90}) in dimension 4, we thus have the following:

\begin{corollary} \label{cor:unique} Let $G$ be a SCY group arising as fundamental group of an infrasolvmanifold. Then the corresponding SCY manifolds are unique up to homeomorphism.
\end{corollary}

Combined with Theorem \ref{thm:morsza} the corollary asserts that, for all known examples, the fundamental group determines in fact the homeomorphism type of SCY $4$--manifolds.

We want to stress   that the results presented above sit at the intersection of symplectic topology and $4$--manifold topology. Namely, they  provide constraints on the type of $M$  that emerge only in this dimension; it is otherwise known by \cite{FP11} that in dimension $6$ for any finitely presented group $G$, there exists a symplectic manifold $Z$ with canonical class $K = 0$ and $\pi_1(Z) = G$ and much latitude on the choice of higher Betti numbers. (In dimension $2$, of course, the only admissible $G$ is $\Z^2$.)

We end with a comment regarding the classification up to diffeomorphism. It is often expected that in dimension $4$  every homeomorphism class of smooth manifolds admits multiple, or even infinite, smoothings. This expectation, however, is founded mostly on our understanding of simply--connected $4$--manifolds. Eilenberg--Maclane spaces may well exhibit a different behavior. Recently, Stern asked in \cite{St12} whether (symplectic) $4$--dimensional Eilenberg--Maclane spaces have at most one smooth structure. The same result of uniqueness up to diffeomorphism for all examples covered by Corollary \ref{cor:unique} would follow, as discussed at the end of Section \ref{sec:topclass}, from a conjecture of Baldridge and Kirk \cite[Conjecture 23]{BK07}, that is motivated by a different circle of ideas.  It is not out of question therefore to expect that Corollary \ref{cor:unique} holds in the smooth category as well. Similarly, all known constructions of exotic simply connected $4$--manifolds fail to produce a symplectic manifold with $K = 0$ nondiffeomorphic to the $K3$ surface.  We can add, for the case of SCY manifolds, a further ingredient:  potentially nondiffeomorphic symplectic smoothings with  $K = 0$ would have to have, even for all finite covers,
 the same Seiberg--Witten invariants. In fact, Taubes' constraints imply that $K = 0$ is the only basic class, and the same situation occurs for all finite covers. This implies in particular that different smoothings would not be distinguishable with any of the known smooth invariants.

\begin{ack} We would like to thank Dieter Kotschick for some useful remarks on the value of the Hausmann--Weinberger and Kotschick invariants of the fundamental group of a surface bundle over a surface. We are also grateful to the referee for carefully reading this paper.\end{ack}

\section{SCY groups} \label{sec:groups}

\subsection{Examples of SCY manifolds with $b_{1} > 0$} \label{sec:examples}

To the best of the authors' knowledge, all known SCY manifolds with $b_{1} > 0$ are \textit{infrasolvmanifolds}. (We refer to  \cite{Hi02} for the rather elaborate definition of infrasolvmanifold and a discussion of their basic properties, from which we will draw in what follows.) In fact, these have appeared in the literature in various forms, and we attempt here to describe them in a uniform way:
\\ -- $T^2$--bundles over $T^2$. These manifolds are symplectic, by \cite{Ge92}, and it is well--known
 (see e.g. \cite{Li06a}) that for these manifolds $K = 0$. $T^2$--bundles over $T^2$ admit a solvmanifold structure, and in fact they  constitute the entire class of solvmanifolds with $2 \leq b_1 \leq 4$ (see e.g. \cite[Proposition 1]{Ha05});
\\ --  $S^1$--bundles over a torus bundle $T^2 \hookrightarrow N^3 \rightarrow S^1$, with the $S^1$--bundle restricting trivially to the fiber $T^2$. These manifolds are symplectic, by \cite{FGM91}, and the fact that  $K = 0$ follows e.g. from \cite[Corollary 2.4]{McS96}. As these manifolds can be described as mapping tori of a selfdiffeomorphism of $T^3$, they admit an infrasolvmanifold structure, see \cite[Chapter 8]{Hi02}. Some of these manifolds have the structure of $T^2$--bundles over $T^2$ (hence solvmanifolds) on the nose while other have an abelian cover that does (see \cite{FV11a}) but it is not clear if they all do;
 \\ -- Cohomologically symplectic infrasolvmanifolds. (Manifolds $M$ for which there exist a class in $H^2(M;\R)$ of positive square.) These are symplectic, by \cite{Ka11}. As they are covered by a $T^2$--bundle over $T^2$, $K$ must be torsion hence (using \cite[Corollary 2.4]{McS96} if $b_{+} = 1$)  they have $K = 0$.

In fact, by the observations above, the class of symplectic infrasolvmanifolds include all known examples of SCY manifolds. (It is not clear to the authors if there exist, in dimension $4$, symplectic infrasolvmanifolds which are not actual solvmanifold; if that does not happen, the list would therefore reduce to $T^2$--bundles over $T^2$.)

In \cite{FV11a} the authors showed that, for $4$--manifolds that admit a circle action, the list is complete.

The manifolds above are all Eilenberg--Maclane spaces, hence their fundamental groups are Poincar\'e duality groups of cohomological dimension $4$. As fundamental groups of infrasolvmanifolds, they are  virtually poly--$\Z$, hence residually finite. As $b_{1}(M) > 0$, these manifolds fiber over $S^1$ (\cite[Lemma 3.14]{Hi02}) hence by \cite{Lu94a} the
$L^2$--Betti number $b_{1}^{(2)}(M)$ vanishes. Two other integral invariants,  the Hausmann--Weinberger invariant $q(G)$ (\cite{HW85}) and the Kotschick invariant $p(G)$ (\cite{Ko93}) defined respectively as
\[ q(G) = \mbox{inf} \{ \chi(X) \}, \ \ p(G) =  \mbox{inf} \{ \chi(X) - |\sigma(X)| \}\] (where the infimum is taken among all $4$--manifolds $X$ with $\pi_{1}(X) = G$) vanish as well (\cite[Corollary 3.12.4]{Hi02}).

\subsection{Bauer--Li constraints}

If we assume the point of view that symplectic infrasolvmanifolds should exhaust the class of SCY manifolds, we must look for evidence that SCY groups satisfy conditions known to hold for fundamental groups of that class of manifolds, such as the vanishing of the invariants discussed in the previous section.
In this section we will use the constraints of Theorem \ref{thm:bauerli} to determine these invariants for SCY groups.

First of all, part (1) of Theorem \ref{thm:bauerli}, together with  L\"uck's Approximation Theorem for $L^2$--invariants (\cite{Lu94}), yields the following:

\begin{lemma} \label{lemma:lueck} Let $G$ be a residually finite SCY group with  $b_{1}(G) > 0$. Then the $L^2$--Betti number $ b_{1}^{(2)}(G) = 0$. \end{lemma}

\begin{proof} As $G$ is residually finite, there exists a  nested cofinal sequence of normal finite index subgroups   $G_{i} \lhd G$. For this sequence, $\lim_{i} [G:G_i] = \infty$. By  \cite{Lu94},  we have \[    b_{1}^{(2)}(G) =  \lim_{i} \frac{b_1(G_i)}{ [G:G_i]}. \]  As $G$ is SCY, so is each finite index subgroup $G_{i} \lhd G$.  It then follows from Theorem \ref{thm:bauerli} that all Betti numbers $b_1(G_i)$ are bounded above by $4$, hence $b_{1}^{(2)}(G) = 0$. \end{proof}

This result allows us to determine immediately the two other integral invariants of $G$, $q(G)$ and $p(G)$.  Lemma \ref{lemma:lueck} implies, by a fairly standard argument, that these invariants vanish:

\begin{proposition}  \label{cor:min}
Let $G$ be a residually finite SCY group with  $b_{1}(G) > 0$.
 Then $q(G) = p(G) = 0$. \end{proposition}

\begin{proof} For any $4$--manifold $X$ with $\pi_{1}(X) = G$
 we have by standard facts of $L^2$-invariants (see e.g. \cite{Lu02}) that
\[ \chi(X) =   2 b_{0}^{(2)}(X) -2  b_{1}^{(2)}(X) +  b_{2}^{(2)}(X) = b_{2}^{(2)}(X).\]
Here we used the fact that $G$ is infinite, which implies  $b_{0}^{(2)}(X) = b_{0}^{(2)}(G) = 0$, and Lemma \ref{lemma:lueck}, which implies $b_{1}^{(2)}(X) = b_{1}^{(2)}(G) = 0$.  Therefore,  $\chi(X) \geq 0$ and $\chi(X) = b_{2}^{(2)}(X) = b_{2,+}^{(2)}(X) + b_{2,-}^{(2)}(X) \geq |b_{2,+}^{(2)}(X) - b_{2,-}^{(2)}(X)| = |\sigma(X)|$, from which $q(G) \geq 0$ and $p(G) \geq 0$ follow. On the other hand, by definition of SCY group  there exists a SCY manifold $M$ with $\pi_{1}(M) = G$ for which the equalities are attained. \end{proof}

A good amount of wishful thinking may give the expectation that the conditions $2 \leq b_1(G) \leq vb_1(G) \leq 4$, $q(G) = p(G) = 0$ characterize symplectic infrasolvmanifold groups, with the exclusion of the somewhat exceptional $\Z^2$ (the $2$--dimensional symplectic infrasolvmanifold group), see below. (The first condition is certainly not sufficient, as the example of $\Z^3$ shows; here, $q(\Z^3) = 2$, see \cite[Lemma 5.1]{Ko94}.) Most likely, this expectation is unfounded: it is however interesting to compare it with an (imperfect) analog of this problem, namely a characterization of $3$--manifolds that fiber over the circle with trivial Euler class in terms of virtual Betti numbers. We have the following.

\begin{proposition}
Let $G$ be a closed orientable $3$--manifold group
 such that $1 \leq b_{1}(G) \leq vb_1(G) \leq 3$;
 then either $G$ is the fundamental group of a (unique up to homeomorphism) $3$--manifold $N$ that fibers over the circle with Euler class $e = 0 \in H^2(N;\Z)$ or $G$ is infinite cyclic. \end{proposition}

\begin{proof} First, assume that $G$ is freely indecomposable.  By Kneser's Theorem, this is equivalent to $G$ being the fundamental group of a prime $3$--manifold. As a consequence $G$ is either finite, infinite cyclic, or the fundamental group of a $3$--dimensional aspherical manifold $N$. We exclude the first condition, for which $b_1(G) = 0$, and proceed with the last one. If $N$ is atoroidal, then by the work of Agol and Wise
 (\cite{Wi12} and \cite{Ag12}) we have  $vb_1(G) = \infty$, see also \cite{AFW12} for references.  Similarly, if $N$ contains an incompressible torus
 which is not a virtual fiber of a fibration, then  $vb_{1}(G) = \infty$ by \cite{Koj87,Lu88}. Therefore,  $N$ must contain a virtual torus fiber, that is promoted to a torus fiber in a cover $\pi : {\tilde N} \to N$. Now, as $b_{1}(G) > 0$, there is a nontrivial class $\phi \in H^{1}(N)$ that lifts to a nontrivial class in $\pi^{*} \phi \in H^{1}({\tilde N})$. As   ${\tilde N}$ is a torus bundle, by standard facts all nontrivial elements of $H^{1}({\tilde N})$ correspond to fibrations, hence $\pi^{*} \phi$ is a fibered class. But then so must $\phi$, namely $N$ itself is a  $T^2$--bundle over $S^1$. To complete the proof, observe that if $G$ is a nontrivial free product of fundamental groups of prime $3$--manifolds, as all $3$--manifold groups are residually finite, it is easy to see that $vb_1(G) = \infty $ on the nose.
\end{proof}

\subsection{Surface bundle groups}

All known SCY manifolds are finitely covered by a $T^2$--bundle over $T^2$, hence the corresponding SCY groups are virtually surface bundle groups. Virtual surface bundle groups are therefore a good starting point to probe how restrictive the constraints of Theorem \ref{thm:bauerli} are. The results on the numerical invariants described above provide a partial answer to this question. For those groups, L\"uck proved in \cite{Lu94a} the vanishing of $b_{1}^{(2)}(G)$, so Lemma \ref{lemma:lueck} is inconclusive. On the other hand, Kotschick proved in \cite[Theorem 3.8]{Ko94} that aspherical manifolds realize the values of $q(p)$ and $p(G)$;  this, and \cite[Theorem 2]{Ko98}, entail that $q(G) \geq p(G) > 0$ whenever $G$ is the fundamental group of a surface bundle over a surface where base and fiber have genus greater than $1$. Proposition \ref{cor:min} applies then to show that such a $G$ is not an SCY group.

 We will see that we can complete this result by verifying (with one exception) that, in the class of groups that arise as fundamental groups of a surface bundle over a surface,  only the groups discussed in  Section \ref{sec:examples} are SCY. This is a consequence of the following result, which has some interest \textit{per se}. We denote by $\Sigma_{g}$ an orientable surface of genus $g$.

\begin{theorem} \label{thm:surfacesurface}
Let $\Sigma_{h} \hookrightarrow M \rightarrow \Sigma_{g}$ be a surface bundle over a surface. Then $\pi_1(M)$ is large if and only if $\mbox{max}(g,h) > 1$.
\end{theorem}

\begin{proof} The ``only if" part of the statement is elementary. When $g > 1$, the ``if" part follows from the surjectivity of the map $\pi_1(M) \to \pi_{1}(\Sigma_{g})$ and the fact that the fundamental group of a surface of genus $g$ admits a surjection onto the free group on $g$ generators. So the interesting case is $\Sigma_{h} \hookrightarrow M \rightarrow T^2$, for $h > 1$. Choose a homology basis $\{s,t \in H_1(T^2)\}$. This determines a marking $S^1_{s} \times S^1_{t}$ of the base of the fibration.
Correspondingly, we can consider the fibration $M \rightarrow S^1_{s}$ with $3$--dimensional fiber $F_t$ that is itself a $\Sigma_{h}$--bundle over $S^1_{t}$, the monodromy of the latter fibration being the restriction of the monodromy of $M$ along $S^{1}_{t}$. The fibration $M \rightarrow S^1_{s}$ arises then as mapping torus of an automorphism $\phi : F_t \to F_t$. By the Nielsen--Thurston classification of automorphisms of surfaces, we can now split the problem in three cases, depending on the isotopy class of the monodromy of $F_t$.
\begin{enumerate}
\item The monodromy of $F_{t} \rightarrow S^1_t$ is pseudo--Anosov. In this case $F_t$ is hyperbolic, whence $Aut(F_t)$ is a finite group so that $\phi \in Aut(F_t)$ can be assumed to have finite order $p$. The cover of $M$ determined by \[ \pi_{1}(M) \rightarrow \pi_{1}(S^1_{s}) \rightarrow \Z_p \] is then the product $F_t \times S^{1}_s$. By the work of Agol and Wise
    (see again \cite{Ag12,Wi12,AFW12}), $\pi_{1}(F_t)$ is large, and the result follows.
\item The monodromy of $F_{t} \rightarrow S^1_t$ is periodic. Denote by $q$ its period. The cover of $M$ determined by \[ \pi_{1}(M) \rightarrow \pi_{1}(S^1_{t}) \rightarrow \Z_q \] is then the product $F_s \times S^{1}_t$, where $F_s$ is the $\Sigma_{h}$--bundle over $S^1_{s}$, the monodromy of the latter fibration being the restriction of the monodromy of $M$ along $S^{1}_{s}$. Invoking the work of Agol and Wise again,  $\pi_{1}(F_s)$ is large and the result follows.
\item The monodromy of $F_{t} \rightarrow S^1_t$ is reducible. 
 This implies (see e.g. \cite[Theorem~2.15]{CSW11}) that, after suitable isotopy of the monodromy, there is an incompressible JSJ torus $T \subset F_{t}$, intersecting each fiber in a disjoint union of circles preserved by the monodromy.
  Recall that $\pi_1(T) \subset \pi_1(F_t)$ is separable by \cite{LN91}, i.e. for any $\gamma \in \pi_{1}(F_{t}) \setminus \pi_1(T)$, there exist an epimorphism to a finite group $\alpha : \pi_{1}(F_{t}) \to Q$ such that $\alpha(\gamma) \notin \alpha(\pi_{1}(T))$. We now need the following virtual extension result:
\begin{claim}
Let $G = \Gamma  \rtimes \Z$ where $\Gamma$ is a finitely generated group, and let $\a : \Gamma \rightarrow Q$ be an epimorphism onto a finite group; then there exists an integer $d$ such that $\a$ extends to the normal
 subgroup $G_d = \Gamma  \rtimes d \Z \lhd G$ with $\a (\gamma,m) = \a(\gamma)$. \end{claim}

This claim is well-known (see e.g. \cite{Bu11}), but for the reader's convenience we give a quick proof.
We denote by $\phi$ the automorphism of $\Gamma$ which corresponds to the semidirect product.
Note that $\hom(\G,Q)$ is a finite set (here we need that $\G$ is finitely generated). It follows that
$\G$ contains only finitely  subgroups such that the quotient is isomorphic to $Q$. Therefore there exists an $r\in \N$ such that $\phi^r(\ker(\a))=\ker(\a)$. In particular $\phi^r$ induces an automorphism of $Q = \G/\ker(\a)$.
Since $\G/\ker(\a)$ is finite there exists an $s\in \N$ such that $\phi^{rs}$ induces the identity on $\G/\ker(\a)$.
We write $d:=rs$.
We will now show that $\ker(\a) \rtimes d\Z$ is a normal subgroup of $\G\rtimes d\Z$,
 which clearly implies the claim.
Let $(g,kd)\in \ker(\a)\rtimes d\Z$ and $(h,ld)\in \G\rtimes \Z$. Then
\[ \ba{rcl} (h,ld)^{-1}(g,kd)(h,ld)&=&(\phi^{-ld}(h^{-1}),-ld)(g,kd)(h,ld)\\
 &=&(\phi^{-ld}(h^{-1}),-ld)(g\phi^{kd}(h),kd+ld)\\
 &=&(\phi^{-1}(h^{-1})\phi^{-ld}(g)\phi^{kd-ld}(h),kd)\\
 &=&(\phi^{-ld}(h^{-1}g\phi^{kd}(h)),kd).\ea \]
Recall that $\phi^{kd}(h)h^{-1}$ lies in $\ker(\a)$, it thus follows that $h^{-1}g\phi^{kd}(h)$ lies in $\ker(\a)$.
Furthermore $\phi^d$ preserves $\ker(\a)$, it thus follows that $\phi^{-ld}(h^{-1}g\phi^{kd}(h))$ also lies in $\ker(\a)$.  This concludes the proof of the claim.

 The claim applies to the fundamental group of a fibration over $S^1$, and asserts that an epimorphism from the fundamental group of the fiber virtually extends, and does so in such a way that the epimorphism is constant along the orbit of the $\Z$--action on the fundamental group of the fiber.

Using the claim for $G = \pi_1(M) = \pi_{1}(F_t) \rtimes \pi_{1}(S^1_{s})$ it now follows from \cite[Proposition 5]{Koj87}
that, possibly after going to  a finite cover of  $M$, we can assume that $T \subset F_{t}$ is nonseparating.
Also, as the number of JSJ tori of  $F_{t}$ is finite (see e.g. \cite[Theorem 3.4]{Bo02})
 we can assume (perhaps up to going to a cover of  $M$ and up to an isotopy of $\phi$), that the automorphism $\phi : F_t \to F_t$  restricts to an automorphism of $T$, which implies that there exists a nonseparating $3$--manifold $\Lambda \subset M$ that restricts to $T$ in each fiber of $M \rightarrow S^1_{s}$. We can now proceed along the lines of \cite[Lemma 2.4]{Lub96} to complete the proof: consider the commutative diagram
\[ \xymatrix{ 1 \ar[r] & \pi_1(T) \ar[d] \ar[r] &\pi_1(\Lambda)\ar[d]\ar[r]  & \Z \ar[d]^{\cong}\ar[r] & 1 \\  1 \ar[r] & \pi_1(F_t \sm \nu  T ) \ar[d] \ar[r] &\pi_1(M \setminus  \nu \Lambda) \ar[d]\ar[r] & \Z \ar[d]^{\cong} \ar[r] & 1 \\  1 \ar[r] & \pi_1(F_t) \ar[r] &\pi_1(M)\ar[r] & \Z \ar[r] & 1 ,}\]
where $\nu$ denotes an open tubular neighborhood. Note that the leftmost and rightmost vertical maps are injective. It follows that the middle vertical maps are injective as well.

 We now pick an identification $\overline{ \nu \Lambda} = \Lambda\times [-1,1]\subset M$. We write $C:=M\sm \nu \Lambda$
 and given $\tau \in [-1,1]$ we write $\Lambda_\tau:=\Lambda\times \tau$.  By the above we have a subgroup inclusion $\pi_1(\Lambda_\tau) \leq \pi_{1}(C) \leq \pi_{1}(M)$  for $\tau = \pm 1$.
 We can now view $\pi_1(M)$ as an HNN extension
 \[ \pi_1(M)=\ll \pi_1(C),t\,|\, \pi_1(\Lambda_{-1})=t\pi_1(\Lambda_1)t^{-1}\rr.\]
 By separability, there exists an epimorphism $\a : \pi_{1}(F_t) \rightarrow Q$ onto a finite group such that $\a(\pi_1(T)) \lneq \a(\pi_{1}(F_{t} \setminus \nu T))$. 
 This entails, using the claim again, that (perhaps on a cover of $M$) we have an epimorphism $\a : \pi_1(M) \rightarrow Q$ such that $\a(\pi_1(\Lambda_{\pm 1})) \lneq \a(\pi_{1}(C))$.
 We now write $A:=\a(\pi_1(C))$ and $B_\pm:=\alpha(\pi_1(\Lambda_{\pm 1})) \lneq A$.
 It  follows from the properties of an HNN extension that $\a$ induces an epimorphism
 \[ \pi_1(M)=\ll \pi_1(C),t\,|\, \pi_1(\Lambda_{-1})=t\pi_1(\Lambda_1)t^{-1}\rr\to
 K:=\ll A,t\,|\, B_-=tB_+t^{-1}\rr.\]
Since $A$ is a finite group it  follows from \cite[Prop~11~p.~120]{Se80} and from \cite[Exercise~3~p.~123]{Se80} that $K$ admits a finite index subgroup which is a nonabelian free group. In particular, $\pi_1(M)$ contains  a finite index subgroup that surjects onto a nonabelian free group, i.e. $\pi_1(M)$ is large.
\end{enumerate}
\end{proof}

We note that the proof Theorem  \ref{thm:surfacesurface} implies, more generally, that $\pi_1(M)$ is large whenever $M$ is a bundle over $S^1$
with fiber a closed irreducible $3$-manifold with non-trivial JSJ decomposition.

Denoting by $\pi_g$ the fundamental group of an orientable surface of genus $g$ we obtain
from Theorem  \ref{thm:surfacesurface}  immediately the following result:

\begin{proposition} Let $1 \rightarrow \pi_{h}  \rightarrow G \rightarrow \pi_{g} \rightarrow 1$ be an extension of a surface group $\pi_g$ by a surface group $\pi_h$. Then if $G$ is a SCY group either $g = h = 1$ or $G$ equals the trivial group or $\Z^2$.
\end{proposition}

All fundamental groups allowed by this proposition are known to be SCY \textit{except} for $G = \Z^2$.  We are not aware of any method to exclude this case. If such a manifold did exist, though, it would be homeomorphic to $S^2 \times T^2$ (see e.g. \cite[Corollary 6.11.1]{Hi02}) but as $K = 0$ it would not be diffeomorphic to it,  a very interesting manifold indeed.

\begin{remark} While it is not clear whether the constraints of Theorem \ref{thm:bauerli} are sufficient to characterize SCY groups, we can ask whether other constrains are available. The authors of this note proved in \cite{FV11a} that  if a symplectic manifold $M$ with trivial canonical class carries a free circle action, then a stronger version of Theorem \ref{thm:bauerli} holds, namely $vb_1(M;\F_{p}) \leq 4$
for any prime $p$ (the conditions $\chi(M) = \sigma(M) = 0$ are trivially satisfied). The interest of this enhanced result is that it allows one to use, in suitable circumstances, information on the growth of  \textit{mod $p$} homology, like the Lubotzky alternative for linear groups (see e.g. ~\cite[Window~9,~Corollary~18]{LS03} and \cite[Theorem 1.3]{La09}), that is stronger than what is available with $\Z$--coefficients. It is not perhaps unreasonable to conjecture therefore that any symplectic $4$--manifold with $K = 0$ satisfies $vb_1(M;\F_{p}) \leq 4$. As a corollary of a result of Lackenby \cite[Theorem 1.10]{La10} we can assert a weak form of this result: Given any epimorphism $\phi : \pi(M) \to \Z$  (which exist as $b_1(M) \geq 1$), the cyclic covers $M_k$ with fundamental group $\phi^{-1}(k\Z)$ must satisfy $\limsup_{k} b_{1}(M_{k};\F_p) < \infty$, otherwise $\pi_1(M)$ would be large.
\end{remark}

\section{The Homeomorphism Type of SCY $4$--manifolds} \label{sec:topclass}

In this section we prove Theorem \ref{thm:main}. We start with the following theorem, which paraphrases  (\cite[Theorem 6]{E97}) of Eckmann:

\begin{theorem}  \textbf{\emph{(Eckmann)}} \label{thm:eckmann} Let $M$ be a $4$--manifold with  $b_{0}^{(2)}(M) = b_{1}^{(2)}(M) = b_{2}^{(2)}(M) = 0$ whose fundamental group $G = \pi_{1}(M)$ satisfies $H^2(G,\Z[G]) = 0$; then either $G$ is virtually infinite cyclic, or $M = K(G,1)$. \end{theorem}

Combining this theorem with Lemma \ref{lemma:lueck} we obtain the following result, equivalent to Theorem \ref{thm:main}:

\begin{theorem} \label{thm:comb}   Let $M$ be an SCY $4$--manifold  with fundamental group $G$.  Assume that $G := \pi_{1}(M)$ is residually finite and satisfies $b_{1}(G) > 0$ and $H^2(G,\Z[G]) = 0$. Then $M = K(G,1)$. \end{theorem}

\begin{proof} As $G = \pi_1(M)$ is infinite, $b_{0}^{(2)}(M)$ vanishes, and so does $b_{1}^{(2)}(M)$, by Lemma  \ref{lemma:lueck}.  The Euler characteristic $\chi(M) = 2 b_{0}^{(2)}(M) -2  b_{2}^{(1)}(M) +  b_{2}^{(2)}(M)$ equals zero by Theorem \ref{thm:bauerli}, hence $b_{2}^{(2)}(M) = 0$ as well. We can apply then Eckmann's Theorem. The case where $G$ is virtually cyclic, i.e. $\Z \lhd_{f.i.} G$, can be excluded: by Theorem \ref{thm:bauerli} we have $b_{1}(G) \geq 2$ and the Betti number is nondecreasing on finite index subgroups.  The statement now follows. \end{proof}

Checking the details of the proof of Theorem \ref{thm:eckmann} it is possible to see that the group  $H^2(G,\Z[G])$ coincides with $\pi_2(M)$, and that this is the only obstruction to $M$ being aspherical. The condition  $H^2(G,\Z[G]) = 0$ is not too restrictive; in particular, it is satisfied by all  (virtual) duality groups of (virtual) cohomological dimension at least $3$, see \cite[Sections~10, 11]{Br94}.

By considering the examples of SCY infrasolvmanifolds we have the corollary:

\begin{corollary} Let $G$ be a SCY group arising as fundamental group of an infrasolvmanifold. Then the corresponding symplectic manifolds are unique up to homeomorphism. \end{corollary}

\begin{proof} Let $M$ be any SCY infrasolvmanifold, and denote $G = \pi_{1}(M)$. The manifold $M$ is an Eilenberg--Maclane space of type $K(G,1)$, from which we deduce that $G$ is a Poincar\'e duality group of cohomological dimension $4$ and $H^2(G,\Z[G]) = 0$. Moreover, $G$ is virtually poly--$\Z$  hence residually finite by a classical result due to Hirsch. We can then apply Theorem \ref{thm:main} to show that any SCY $X$ with fundamental group equal to $G$ is homotopic to $M$. But with virtually poly--$\Z$ groups we can apply to machinery of \cite{FQ90} to deduce that $X$ is actually homeomorphic to $M$, see e.g. \cite[Theorem 2.16]{FJ90}. \end{proof}

Baldridge and Kirk \cite[Conjecture 23]{BK07} formulated the following conjecture:
Let  $X$ and $Y$ be 4-manifolds which  realize the minimum of the Hausmann--Weinberger invariant of a group $G$.
If $X$ and $Y$  have equivalent intersection forms, then they are in fact diffeomorphic.

We will now argue that this conjecture implies in particular that SCY manifolds  with a given group $G$ are  unique up to diffeomorphism.
We first show that
 for a SCY manifold, the intersection form is determined by the fundamental group $G$. In fact  by Theorem \ref{thm:bauerli}  the rank is determined by $b_1(G)$ and the signature is always zero (hence the form is indefinite), and as the characteristic element $K$ vanishes, the parity is even.
The argument is thus completed by the fact that a SCY manifold realizes the minimum value $q(G) = 0$.



\end{document}